\documentclass[12pt,a4paper]{amsart} 
\usepackage[utf8]{inputenc}
\usepackage[english]{babel}

\usepackage{amsmath}
\usepackage{amsfonts}
\usepackage{amssymb}
\usepackage{mathrsfs}
\usepackage{color}

\usepackage{amsthm}
\newtheorem{prop}{Proposition} [section]

\newtheorem{thm}{Theorem} [section]

\usepackage{graphics}
\usepackage{epstopdf}
\usepackage{epsfig}

\usepackage{url}
\usepackage{hyperref}

\usepackage[a4paper,top=2.5cm,bottom=2.8cm,
left=3.2cm,right=3.2cm]{geometry}

\newcommand\Al{\alpha}
\newcommand\B{\beta}
\newcommand\De{\delta}
\newcommand\G{\gamma}
\newcommand\Ep{\varepsilon}

\newcommand\CH{\mathcal{H}}

\newcounter{num} 
\newcommand{\Fg}[1][]{\thenum}

\definecolor{pink}{rgb}{1.0,.4,.0} 
\definecolor{green2}{rgb}{.7,.7,.7} 
\definecolor{pink2}{rgb}{1.,.8,.1} 
\definecolor{blue2}{rgb}{.6,.5,1} 
\definecolor{blue3}{rgb}{.4,.6,.2} 
\definecolor{pink3}{rgb}{.7,.4,.4} 
\definecolor{mag}{rgb}{.6,.0,1.0} 
\definecolor{rd}{rgb}{1.,.0,.0} 

\begin{document}
\bigskip
\bigskip\bigskip

\begin{center}
{\Large \textbf{Haga's theorems in paper folding and 
related theorems in Wasan geometry Part 1}} \\
\bigskip

\bigskip
\textsc{Hiroshi Okumura} \\

\bigskip\bigskip
\end{center}

\textbf{Abstract.} 
Haga's fold in paper folding is generalized. Recent generalization 
of Haga's theorems and problems in Wasan geometry involving Haga's 
fold are also generalized. 



\bigskip
\bigskip

\section{Introduction}

Let $ABCD$ be a piece of square paper with a point $E$ on the side $DA$. 
We fold the paper so that the corner $C$ coincides with $E$ and the side 
$BC$ is carried into $B'E$, which meets the side $AB$ in a point $F$. We 
call this a Haga's fold. Unifying Haga's theorems in 
paper folding in \cite{Haga}, we obtained the following theorem \cite{OKFG14} 
(see Figure \ref{fhaga0}). 

\begin{center}
\includegraphics[clip,width=51mm]{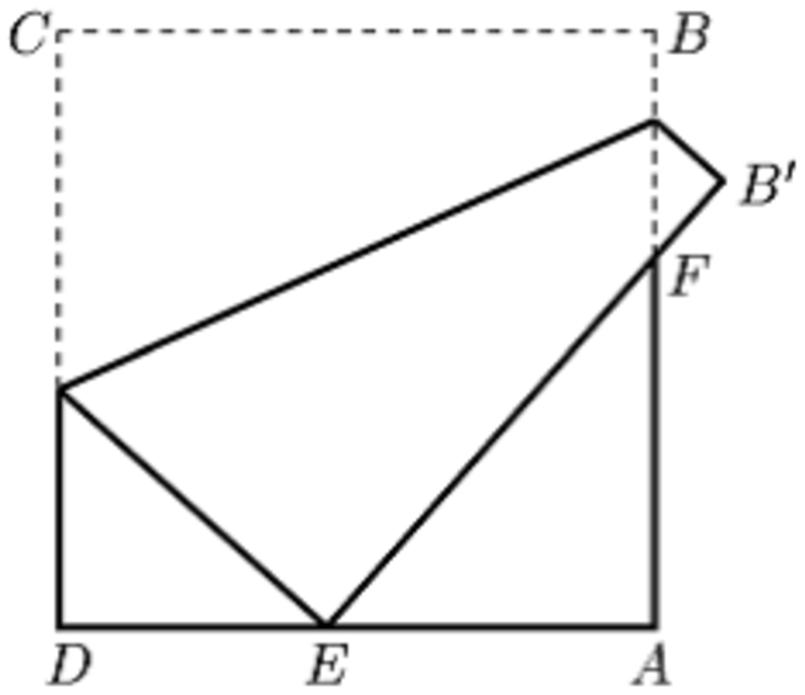}\refstepcounter{num}\label{fhaga0}\\
\medskip 
Figure \Fg .
\end{center} 

\begin{thm}\label{tuh}
The relation $|AE||AF|=2|DE||BF|$ holds for Haga's fold. 
\end{thm}

In this paper we generalize Haga's fold and show that Theorem \ref{tuh} 
holds for the generalized Haga's fold. 
Haga's fold was also considered in Wasan geometry, and one of the most 
famous results says that the inradius of the right triangle $AEF$ equals 
the overhang $|B'F|$. We also generalize this result and show that not 
only the incircle of $AEF$, but the excircles of it play important roles. 
There are three similar right triangles in the figure made by the 
generalized Haga's fold. We show that there are many simple relationships 
between the incircles and the excircles of those triangles.  

\section{Preliminaries}

In this section we summarize several results, which are used in later 
sections. Let $a=|BC|$, $b=|CA|$ and $c=|AB|$ for a triangle $ABC$. 

\medskip
\begin{minipage}{0.33\hsize}  
\begin{center}
\vskip-2mm
\includegraphics[clip,width=48mm]{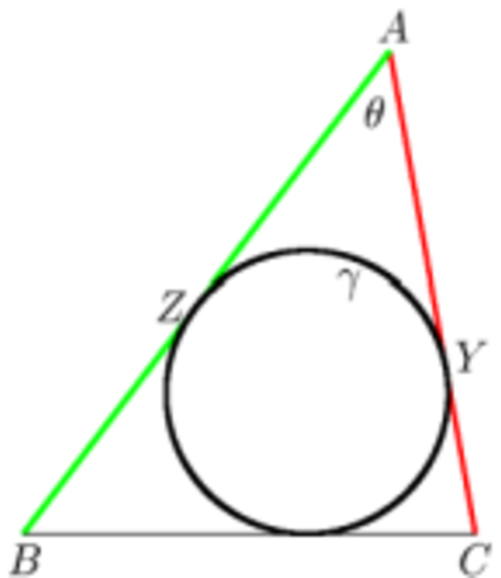}\refstepcounter{num}\label{f3in}\\
\vskip-1.5mm
Figure \Fg . 
\end{center} 
\end{minipage}
\begin{minipage}{0.32\hsize}
\begin{center}
\vskip2mm
\includegraphics[clip,width=48mm]{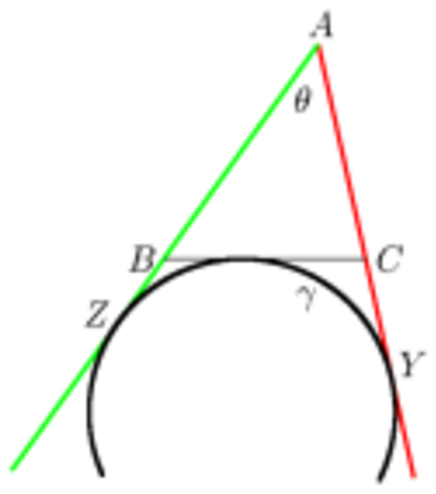}\refstepcounter{num}\label{f3ex}\\
\vskip0mm
Figure \Fg . 
\end{center}
\end{minipage}
\begin{minipage}{0.35\hsize}
\begin{center} 
\vskip4mm
\includegraphics[clip,width=41mm]{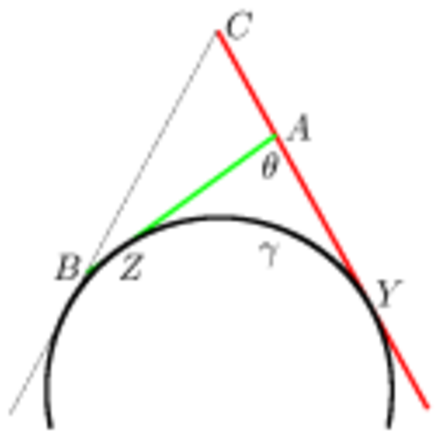}\refstepcounter{num}\label{f3ex2}\\
\vskip1mm
Figure \Fg . 
\end{center} 
\end{minipage} 
\medskip

\begin{thm}\label{tsin}
Let $\G$ be the incircle or one of the excircles of a triangle $ABC$ 
touching the lines $AB$ and $CA$ at points $Z$ and $Y$, respectively. 
If $\theta$ is the angle subtended by $\G$ from $A$, we get 
$$
\sin^2\frac{\theta}{2}=\frac{|BZ||CY|}{bc}.
$$
\end{thm}

\begin{proof}
There are two cases to be considered: 
(i) $\G$ is the incircle of $ABC$ or $\G$ touches the side 
$BC$ from the side opposite to $A$, i.e., $\theta=\angle CAB$ (see Figures 
\ref{f3in}, \ref{f3ex}). (ii) $\G$ touches 
the side $AB$ (resp. $CA$) from the side opposite to $C$ (resp. $B$), i.e., 
$\theta+\angle CAB=180^{\circ}$ (see Figure \ref{f3ex2}). We prove the case 
(ii). Let us assume that $\G$ touches the side $AB$ from the side opposite 
to $C$. Then 
\begin{eqnarray}\nonumber
\sin^2\frac{\theta}{2}&=&\frac{1-\cos \theta}{2}=\frac{1+\cos \angle CAB}{2}
=\frac{1+(b^2+c^2-a^2)/(2bc)}{2} \\ \nonumber
&=&\frac{(b+c)^2-a^2}{4bc}=\frac{(a+b+c)(b+c-a)}{4bc}=\frac{|BZ||CY|}{bc}, 
\end{eqnarray}
because $a+b+c=a+b+|BZ|+|AY|=(a+|BZ|)+(b+|AY|)=2|CY|$, and $b+c-a=
(|CY|-|AZ|)+(|AZ|+|BZ|)-(|CY|-|BZ|)=2|BZ|$. 
The rest of the theorem can be proved in a similar way. 
\end{proof}

Let $\G$ be the incircle of $ABC$ with radius $r$, and let $\G_a$ be the 
excircle of $ABC$ touching $BC$ from the side opposite to $A$ with radius 
$r_a$. 

\begin{thm}\label{tgz}
If $ABC$ is a right triangle with right angle at $A$, the following 
statements hold. \\
\noindent{\rm (i)} The remaining common tangent of each of the pairs 
$\G_a$ and $\G_b$, $\G_a$ and $\G_c$, $\G$ and $\G_b$, $\G$ and $\G_c$ 
is perpendicular to $BC$ \cite{OKGZ90}, and $r_a=r+r_b+r_c$ holds 
\cite{hansen}. \\
\noindent{\rm (ii)} $rr_a=r_br_c$ \cite{hansen}. \hskip40mm 
\noindent{\rm (iii)} $r+r_b=b$ and $r+r_c=c$. 
\end{thm}

\begin{proof} 
$r+r_b=(-a+b+c)/2+(a+b-c)/2=b$. This proves (iii). 
\end{proof}

\section{Generalized Haga's fold}\label{sghf}

In this section, we generalize Haga's fold and Theorem \ref{tuh}. 
Haga's fold is uniquely determined if we fix a point $E$ on 
the side $DA$ for a square $ABCD$. We now consider $E$ is a point on the 
{\sl line} $DA$ for a square $ABCD$ and assume that $m$ is the perpendicular 
bisector of the segment $CE$, $B'$ is the reflection of $B$ in $m$, and $F$ 
is the point of intersection of the lines $AB$ and $B'E$ if they intersect, 
where we define $B=B'=F$ in the case $E=A$. The figure consisting of $ABCD$ 
and the points $B'$, $E$ and $F$ (if exists) is denoted by $\CH$. 
If $ABCD$ is a piece of square paper and $m$ passes through the inside of 
$ABCD$, we can really fold the paper with crease line $m$ so that $BC$ is 
carried into $B'E$ (see Figures from \ref{fcase3} to \ref{fcase71}). We 
thus call $m$ the crease line.  

\medskip\medskip
\begin{minipage}{0.5\hsize} 
\begin{center} 
\includegraphics[clip,width=57mm]{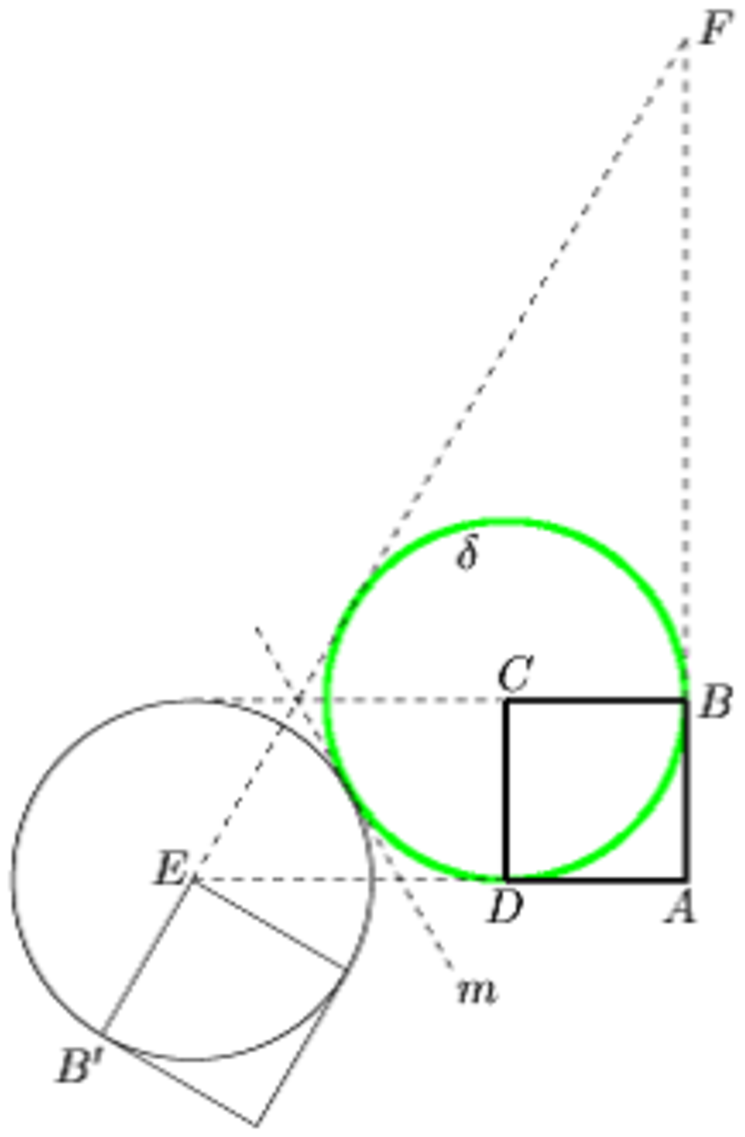}\refstepcounter{num}\label{fcase1}\\
Figure \Fg : (h1) 
\end{center} 
\end{minipage}
\begin{minipage}{0.5\hsize}
\begin{center} 
\vskip12.1mm
\includegraphics[clip,width=57mm]{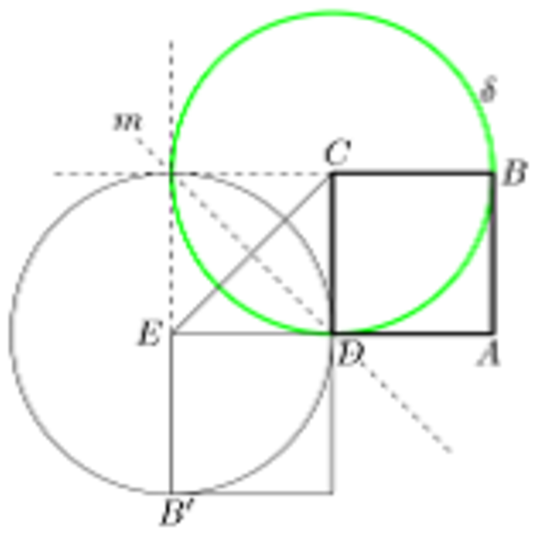}\refstepcounter{num}\label{fcase2}\\
 \vskip12.1mm
Figure \Fg : (h2)
\end{center}  
\end{minipage}
\medskip

\medskip
\begin{minipage}{0.5\hsize} 
\begin{center} 
\includegraphics[clip,width=42mm]{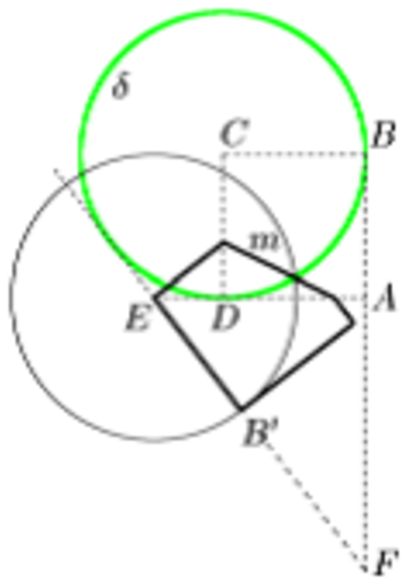}\refstepcounter{num}\label{fcase3}\\
Figure \Fg : (h3) 
\end{center} 
\end{minipage}
\begin{minipage}{0.5\hsize}
\begin{center} 
\vskip8.4mm
\includegraphics[clip,width=41mm]{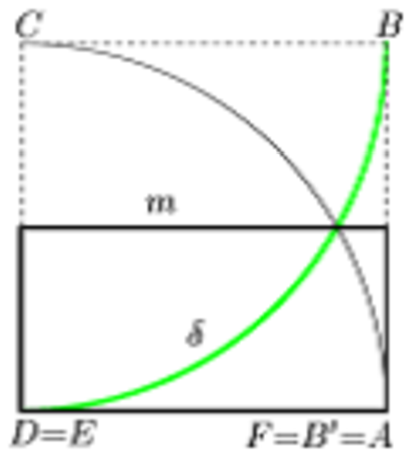}\refstepcounter{num}\label{fcase4}\\
 \vskip5mm
Figure \Fg : (h4)
\end{center}  
\end{minipage}

\medskip
Let $b=|DE|$, $d=|AB|$. 
There are seven cases to be considered for $\CH$ (see Figures 
\ref{fcase1} to \ref{fcase73}):  \\
(h1) $D$ lies between $E$ and $A$ and $b>d$, \hskip1.6mm 
(h2) $D$ is the midpoint of $EA$, \\ 
(h3) $D$ lies between $E$ and $A$ and $b<d$,\ \ (h4) $D=E$, \\
(h5) $E$ lies between $D$ and $A$, \hskip20.3mm (h6) $E=A$, \\
(h7) $A$ lies between $D$ and $E$.  \\

\medskip
\begin{minipage}{0.55\hsize} 
\begin{center} 
\vskip1mm
\includegraphics[clip,width=50mm]{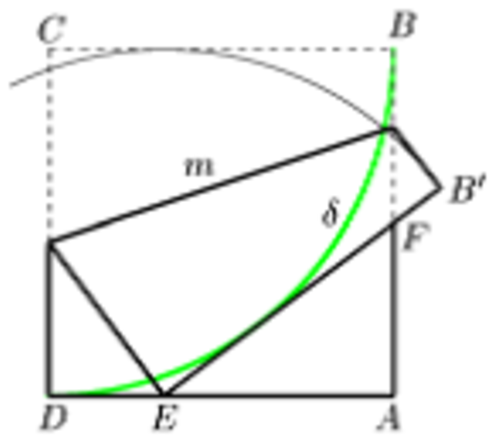}\refstepcounter{num}\label{fcase5}\\
\vskip2mm
Figure \Fg : (h5) 
\end{center} 
\end{minipage}
\begin{minipage}{0.45\hsize}
\begin{center} 
\vskip0.5mm
\includegraphics[clip,width=41mm]{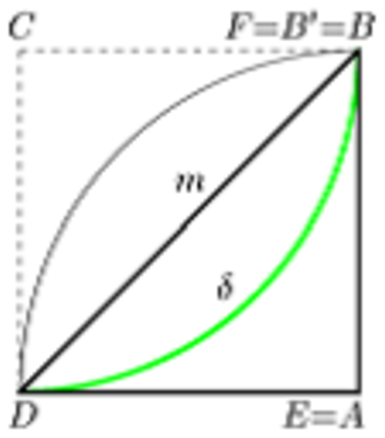}\refstepcounter{num}\label{fcase6}\\
\vskip0.5mm
Figure \Fg : (h6)
\end{center}  
\end{minipage}

\medskip
\begin{minipage}{0.45\hsize}
\begin{center} 
\includegraphics[clip,width=51mm]{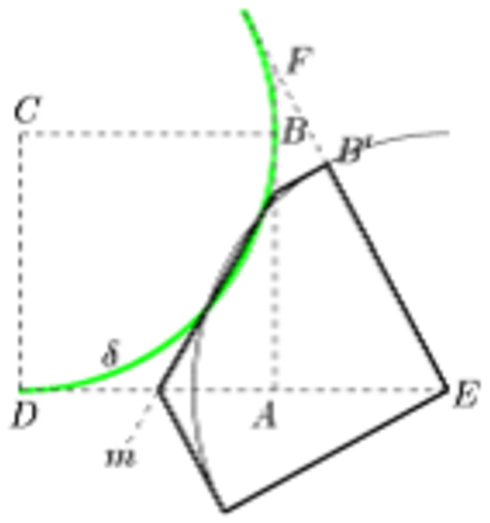}\refstepcounter{num}\label{fcase71}\\
Figure \Fg : (h7), $|AE|<\sqrt{2}d$
\end{center}  
\end{minipage}
\begin{minipage}{0.55\hsize} 
\begin{center} 
\includegraphics[clip,width=59mm]{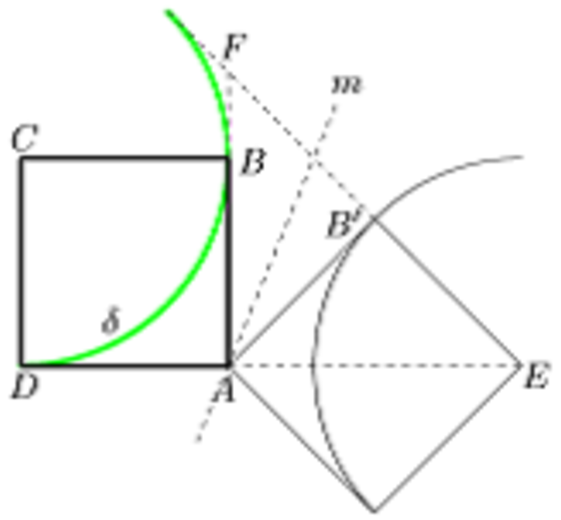}\refstepcounter{num}\label{fcase72}\\
Figure \Fg : (h7), $|AE|=\sqrt{2}d$
\end{center}  
\end{minipage}

\medskip
\begin{center} 
\includegraphics[clip,width=88mm]{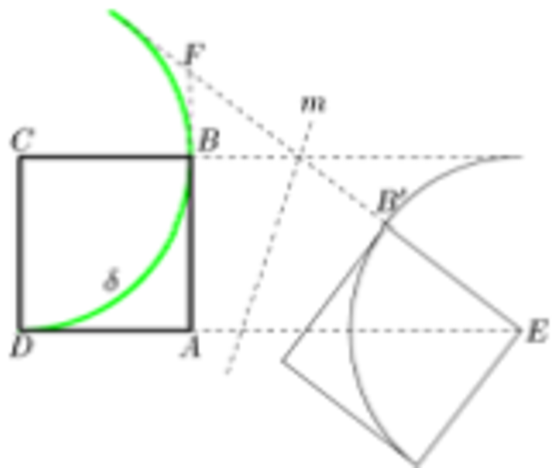}\refstepcounter{num}\label{fcase73}\\
\vskip-1mm
Figure \Fg : (h7), $|AE|>\sqrt{2}d$
\end{center}  

\medskip
The point $F$ does not exists only in the case (h2). Let $c=|BF|$, 
if $F$ exist. Let $\De$ be the circle with center $C$ passing through $B$. 

\begin{prop}\label{pt} The followings hold for $\CH$. \\
\noindent{\rm (i)} 
The line $B'E$ touches the circle $\De$. \\
\noindent{\rm (ii)} 
$|EF|=\left\{ 
\begin{array}{ll}\nonumber 
    c-b & \hbox{in the case {\rm (h3)}}, \\
    b+c & in\ the\ cases\ {\hbox {\rm(h1), (h4), (h5), (h6)}}, \\
    b-c & \hbox{in the case {\rm (h7)}}. \\
\end{array} \right.
$
\end{prop}

\begin{proof}
The circle of radius $d$ with center $E$ touches the line $BC$ (see 
Figures \ref{fcase1} to \ref{fcase73}). While the reflections of this 
circle and the line $BC$ in the crease line $m$ are $\De$ and the line 
$B'E$. This proves (i). The part (ii) follows from (i). 
\end{proof}

Notice that $b=0$ and $|AF|=0$ are equivalent, also $c=0$ and 
$|AE|=0$ are equivalent. Therefore $|AE||AF|=0$ if and only if $bc=0$. 
The next theorem is a generalization of the result in \cite{OKFG14}. 

\begin{thm}\label{ghaga} 
If the lines $AB$ and $B'E$ intersect, then $|AE||AF|=2bc$ for $\CH$. 
\end{thm}

\begin{proof} 
If $|AE||AF|=0$, $bc=0$. Let $|AE||AF|\not=0$. Considering the points 
$E$, $F$ as $B$, $C$ in Theorem \ref{tsin} and $D$, $B$ as $Z$, $Y$, 
respectively, we have 
$$
\frac{bc}{|AE||AF|}=\sin^2 45^{\circ}=\frac{1}{2}. 
$$
\end{proof}

\section{Incircle and excircles of $AEF$} \label{sie}

From now on we assume that the lines $AB$ and $B'E$ intersect. Therefore 
the case (h2) is not considered. We consider the result in Wasan geometry 
saying that the inradius of $AFE$ equals the overhang $|B'F|$ in the case 
(h5) \cite{Fuku}, \cite{ito}, \cite{Tkd}, \cite{ukyojutsu}. The result has 
been quoted from a sangaku problem dated 1893 in \cite{Fuku}, but the book 
\cite{Tkd} is older. Several new results related to this are also given 
in this section. We now call (h1), (h3), (h5) and (h7) the ordinary cases, 
and (h4) and (h6) the degenerate cases. 

We define three circles (see Figures \ref{fcase1c} to \ref{fcase71c}). 
Let $\Al(\not=\De)$ be the incircle or one of the excircles of $AEF$ 
with center lying on the line $AC$ in the ordinary cases, the point $A$ 
in the degenerate cases. 
Let $\B(\not=\Al)$ be the incircle or one of the excircles of $AEF$ with 
center lying on the line joining $E$ and the center of $\Al$ in the ordinary 
cases, the point $A$ in the case (h4), and the reflection of the circle 
$\De$ in the line $AB$ in the case (h6).
Let $\G(\not=\Al)$ be the incircle or one of the excircles of $AEF$ with 
center lying on the line joining $F$ and the center of $\Al$ in the 
ordinary cases, the reflection of $\De$ in the line $DA$ in the case (h4), 
and the point $A$ in the case (h6). Let $a=|B'F|$. 

\begin{thm} \label{tep1} 
The radii of the circles $\Al$, $\B$ and $\G$ equal $a$, $b$ and $c$, 
respectively. 
\end{thm}

\begin{proof} We prove that the radius of $\Al$ equals $a$. The degenerate 
cases are trivial. The proof in the ordinary cases is similar to that for 
Haga's fold given in \cite{Hons}. We prove the case (h1) (see Figure 
\ref{fcase1c}). By Proposition \ref{pt}(ii), the radius of $\Al$ equals 
$$
\frac{|AE|+|AF|+|EF|}{2}=\frac{(b+d)+(c+d)+|EF|}{2} 
=d+|EF|=a.
$$
The other cases can be proved in a similar way. We now prove that the radius 
of $\B$ equals $b$. The degenerate cases are obvious. In the case (h1) (see 
Figure \ref{fcase1c}), the radius of $\B$ equals 
$$
\frac{|AE|-|AF|+|EF|}{2}=\frac{(b+d)-(c+d)+(b+c)}{2}=b 
$$
by Proposition \ref{pt}(ii). The rest of the theorem can be proved 
in a similar way. 
\end{proof}

\begin{center} 
\vskip3mm
\includegraphics[clip,width=99mm]{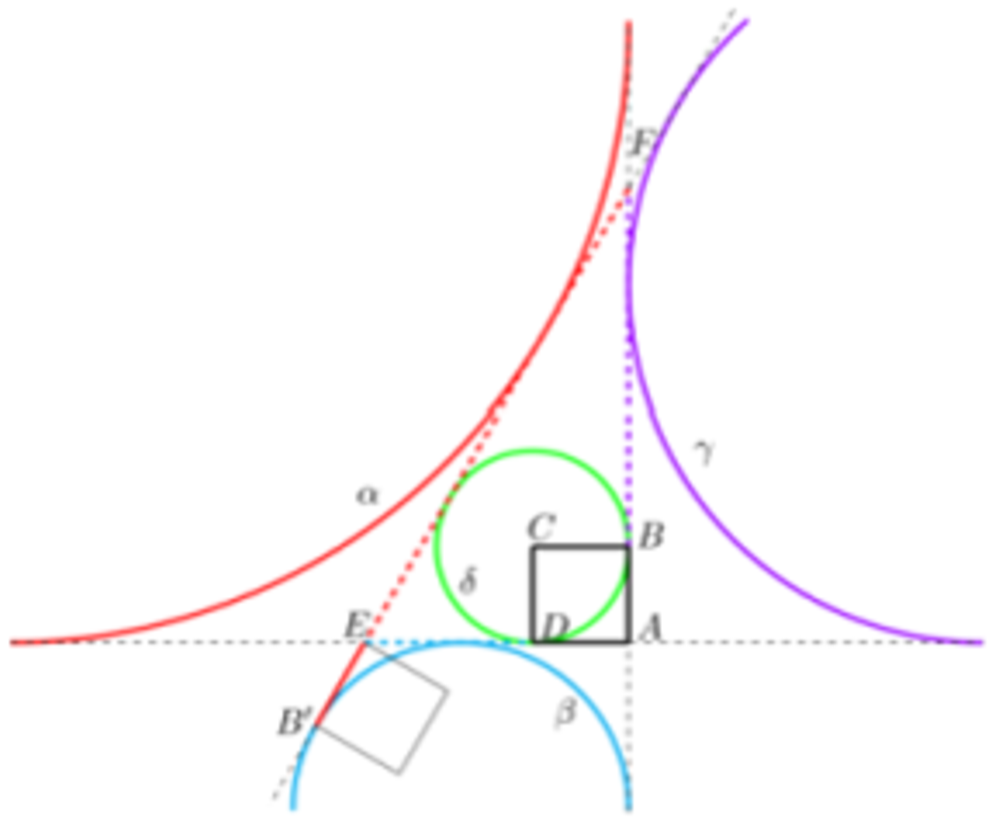}\refstepcounter{num}\label{fcase1c}\\
\vskip8mm
Figure \Fg : (h1)		
\end{center}  
\medskip

\begin{center} 
\includegraphics[clip,width=95mm]{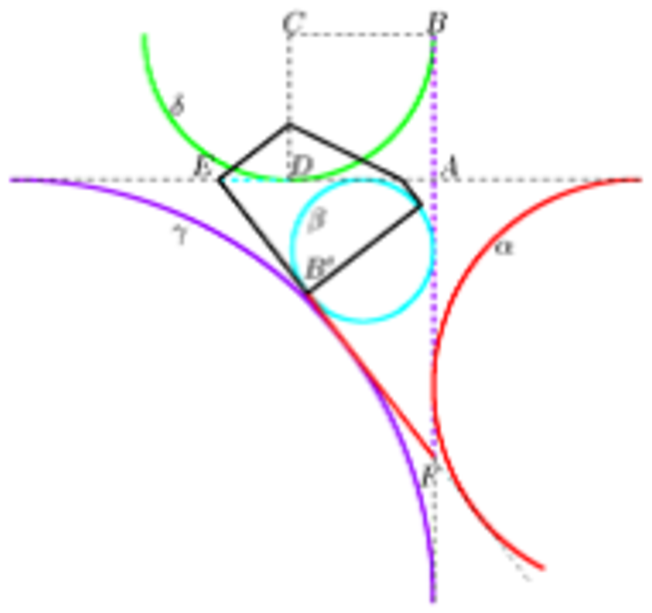}\refstepcounter{num}\label{fcase3c}\\
Figure \Fg : (h3)
\end{center}

\begin{minipage}{0.4\hsize}
\begin{center} 
\vskip0mm
\includegraphics[clip,width=55mm]{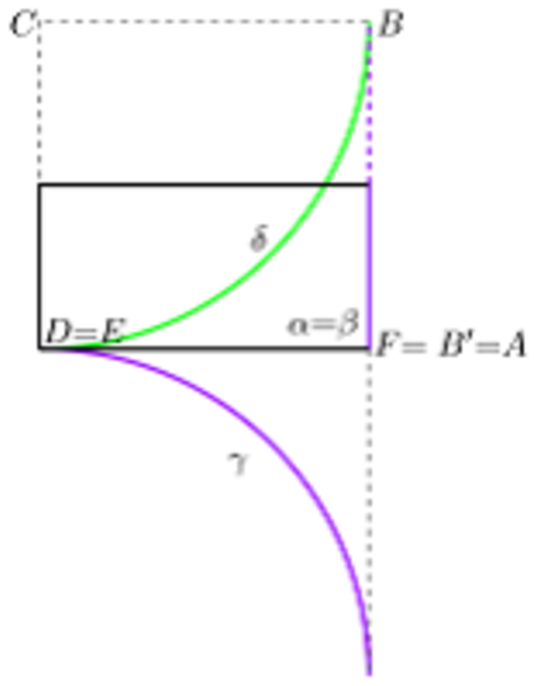}\refstepcounter{num}\label{fcase4c}\\
\vskip0mm
Figure \Fg : (h4)
\end{center}  
\end{minipage}
\begin{minipage}{0.6\hsize}
\begin{center} 
\includegraphics[clip,width=65mm]{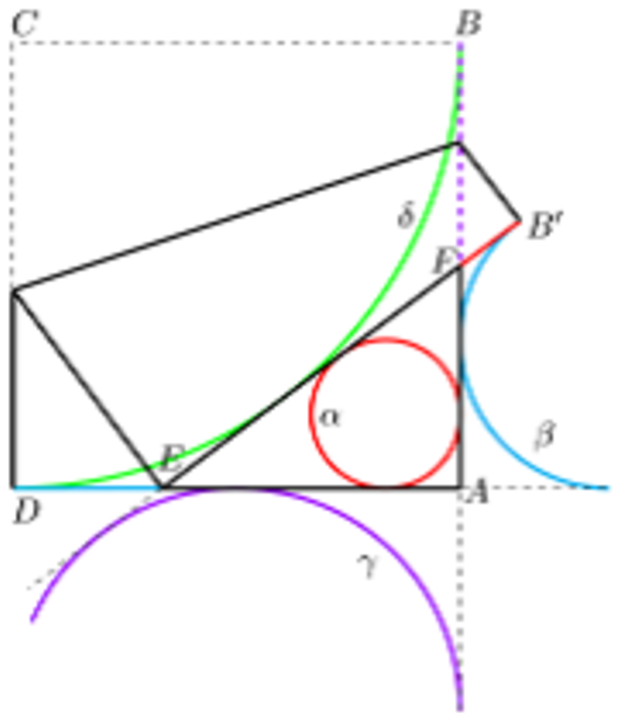}\refstepcounter{num}\label{fcase5c}\\
Figure \Fg : (h5)
\end{center}  
\end{minipage}

\begin{minipage}{0.4\hsize}
\begin{center} 
\vskip4mm 
\includegraphics[clip,width=55mm]{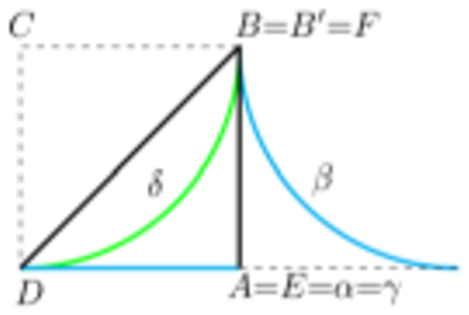}\refstepcounter{num}\label{fcase6c}
\vskip0mm 
Figure \Fg : (h6)
\end{center}  
\end{minipage}
\begin{minipage}{0.6\hsize}
\begin{center} 
\includegraphics[clip,width=69mm]{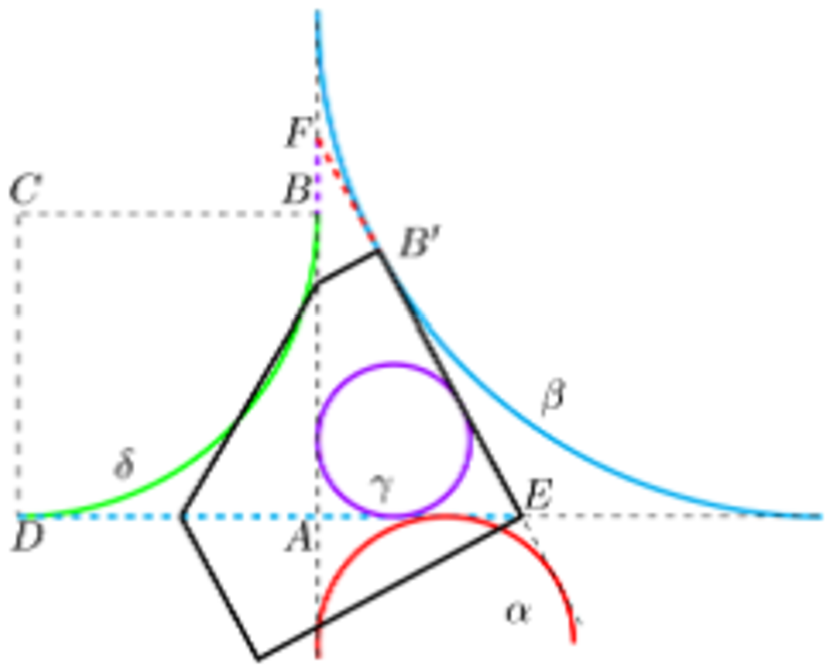}\refstepcounter{num}\label{fcase71c}\\
Figure \Fg : (h7)
\end{center}  
\end{minipage}
\medskip

The set $\{\Al, \B, \G, \De \}$ consists of the incircle and the 
excircles of the triangle $AEF$, if $AEF$ is proper (see Figures \ref{fcc1} 
to \ref{fcc71}). Therefore $a=b+c+d$, $c=a+b+d$, $d=a+b+c$, $b=a+c+d$ 
hold in each of the cases (h1), (h3), (h5), (h7), respectively by Theorem 
\ref{tgz}. The equation $d=a+b+c$ also hold in the degenerate cases. 
Hence we get (i) of the following theorem. The part (ii) follows from 
Theorem \ref{tgz}(ii).

\begin{thm} \label{teqs}The following relations hold. \\
\noindent {\rm (i)} 
$d=\left\{ 
\begin{array}{llll}\nonumber
\hskip3.25mm a-b-c & \hbox{in the case {\rm (h1)}}, \\ 
 -a-b+c & \hbox{in the case {\rm (h3)}}, \\ 
\hskip3.25mm a+b+c & \hbox{in the case {\rm (h4), (h5), (h6)}}, \\ 
 -a+b-c            & \hbox{in the case {\rm (h7)}}. \\ 
\end{array} \right. \\
$
\\
\noindent {\rm (ii)} 
$ad=bc$.
\end{thm}

\begin{thm}\label{tbp}
The circle $\B$ touches $B'E$ at the point $B'$ or coincides with $B'$. 
\end{thm}

\begin{proof} 
It is sufficient to consider the ordinary cases. 
In the cases (h1) and (h3), the distance between $E$ and the point of 
contact of $\B$ and the line $DA$ equals 
$|AE|-b=(b+d)-b=d=|B'E|$. Hence 
$B'$ is the point of contact of $\B$ and $B'E$. 
The other cases can be proved similarly. 
\end{proof}

If a triangle $AEF$ with right angle at $A$ is given, a square $ABCD$ 
forming the configuration $\CH$ can be constructed as follows. 
Let $\Al$ be the incircle or one of the excircles of $AEF$. Then let 
$\De (\not=\Al)$ be the incircle or one of the excircles of $AEF$ 
with center on the line joining $A$ and the center of $\Al$. Let $C$ be 
the center of $\De$ and let $B$ and $D$ be the feet of perpendiculars 
from $C$ to the lines $AF$ and $AE$, respectively. Then $ABCD$ is the 
desired square. Since there are four choices for the circle $\Al$, we 
can construct four squares from $AEF$. 


\begin{center} 
\includegraphics[clip,width=110mm]{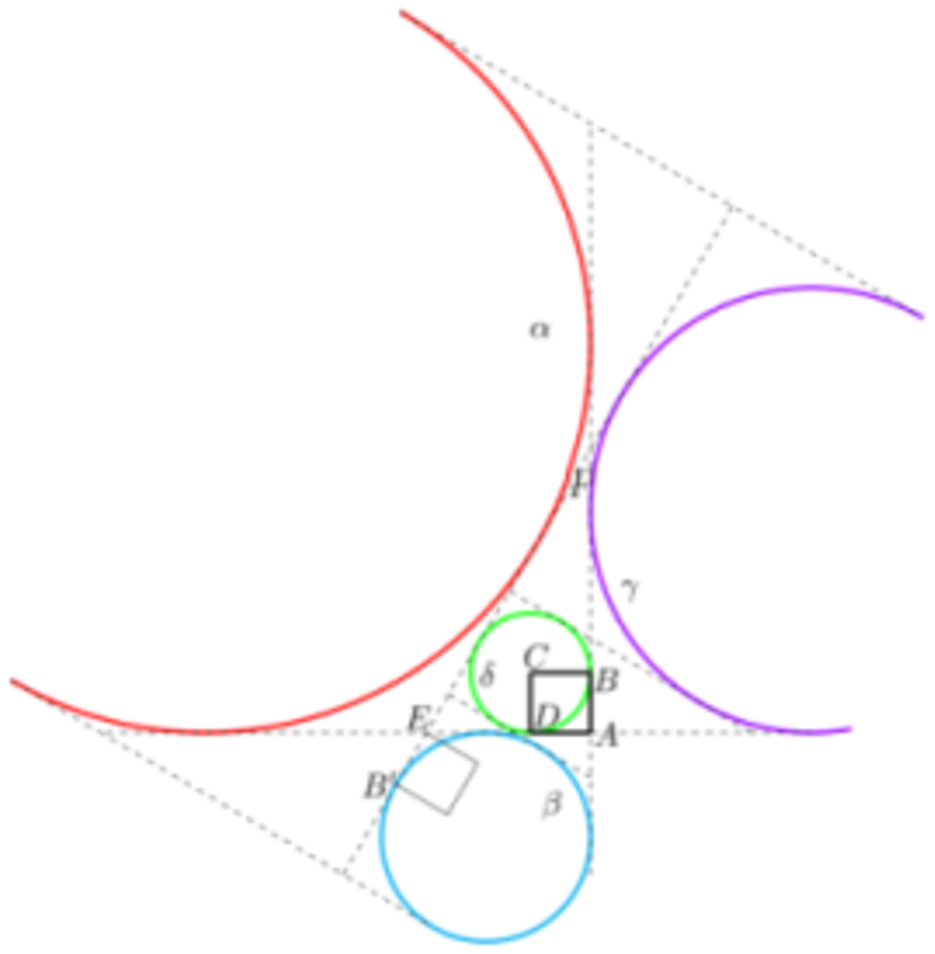}\refstepcounter{num}\label{fcc1}\\
\vskip1mm
Figure \Fg : $a=b+c+d$ in the case (h1)
\end{center}  

\medskip
\begin{center} 
\includegraphics[clip,width=95mm]{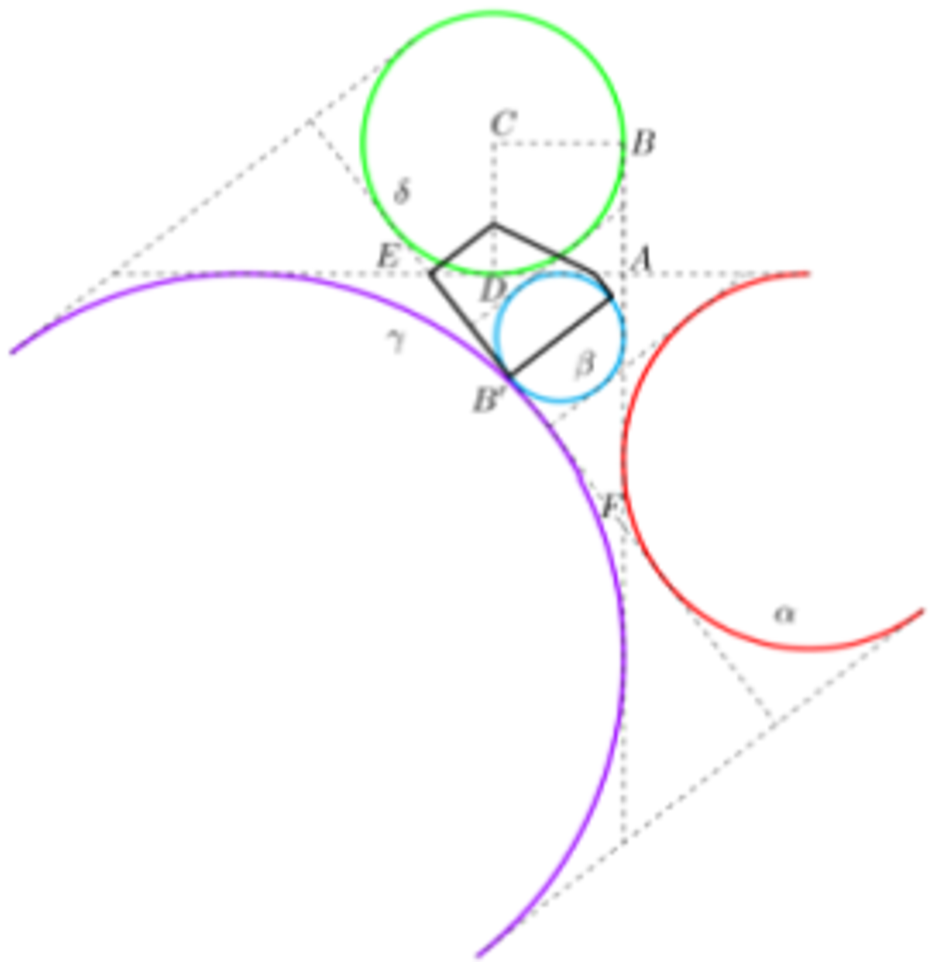}\refstepcounter{num}\label{fcc3}\\
\vskip-1mm
Figure \Fg : $c=a+b+d$ in the case (h3)
\end{center}  

\medskip
\begin{center} 
\includegraphics[clip,width=105mm]{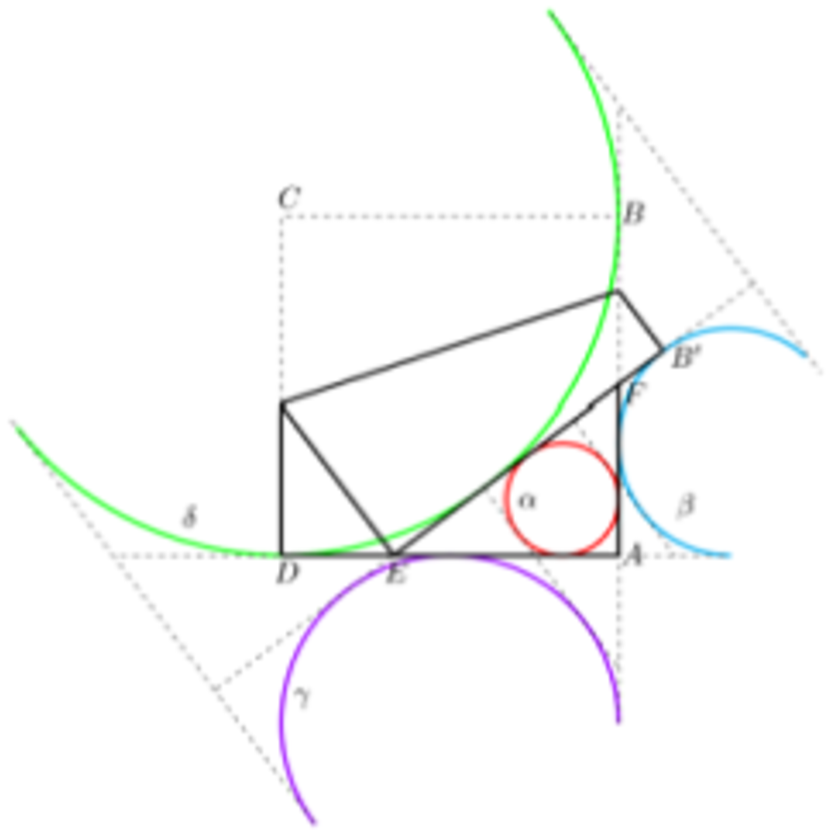}\refstepcounter{num}\label{fcc5}\\
Figure \Fg : $d=a+b+c$ in the case (h5)
\end{center}  

\medskip
\begin{center} 
\includegraphics[clip,width=105mm]{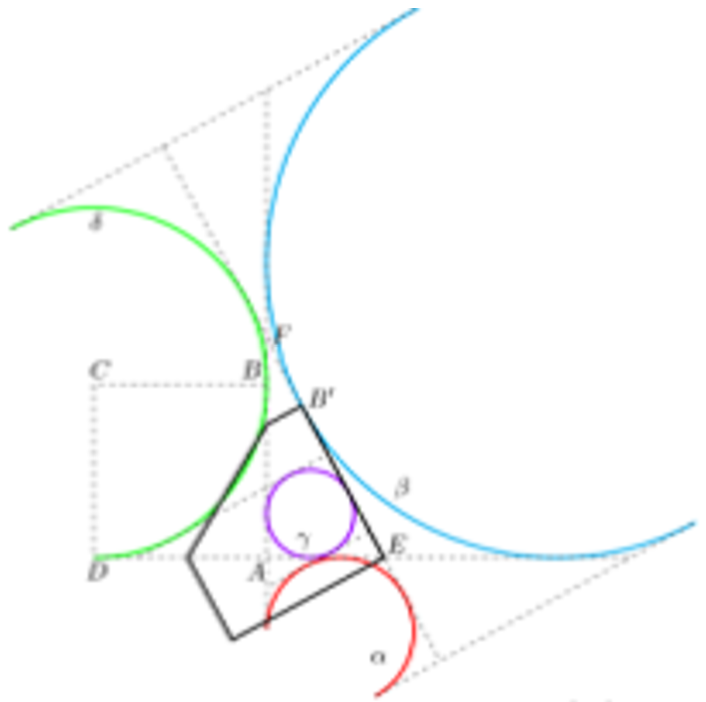}\refstepcounter{num}\label{fcc71}\\
\vskip-1mm
Figure \Fg : $b=a+c+d$ in the case (h7)
\end{center}  
\medskip

\section{Circumcircle of $AEF$}\label{scir}

The circumradius of $AEF$ equals the inradius of the triangle made by 
the lines $BC$, $CD$ and $EF$ in the case (h5) (see Figure \ref{fe2}) 
\cite{OKMM14}. In this section we generalize this result for $\CH$. 
Let $\Ep_1$ be the circle touching the lines $BC$, $CD$ and $EF$ with 
center lying on the lines $m$ and $AC$ (see Figures \ref{fep4h1} to 
\ref{fe3}). The definition shows that $\Ep_1$ is the incircle of $ABCD$ 
in the degenerate cases. Let $r_1$ be the radius of $\Ep_1$. 

\medskip
\begin{minipage}{0.5\hsize} 
\begin{center}
\includegraphics[clip,width=55mm]{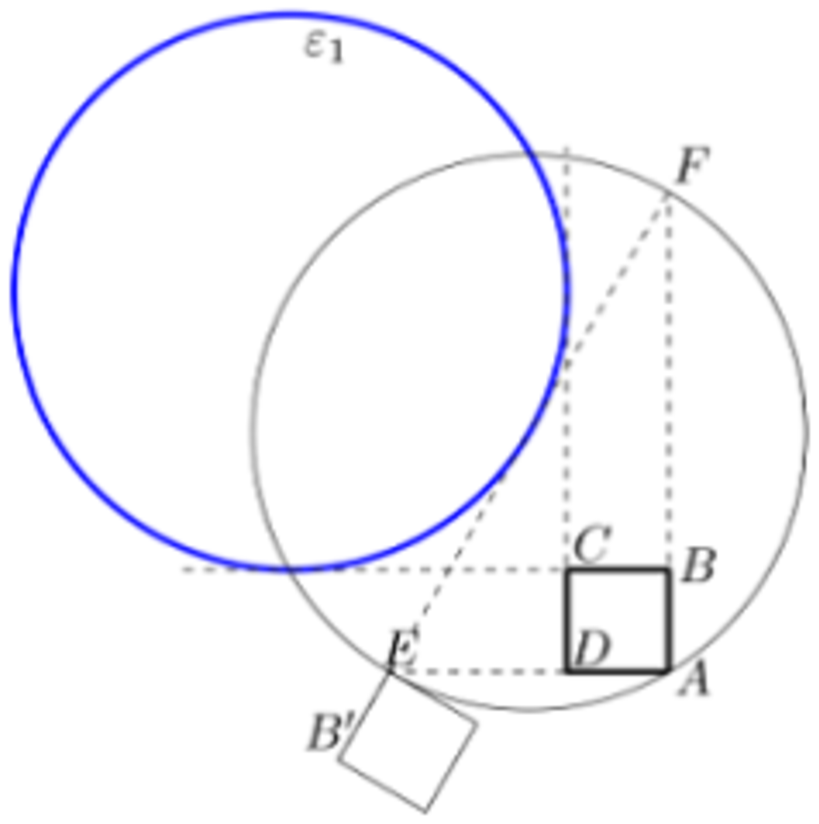}\refstepcounter{num}\label{fep4h1}\\
\vskip2.6mm
Figure \Fg : (h1) 
\end{center} 
\end{minipage}
\begin{minipage}{0.5\hsize}
\begin{center}
\includegraphics[clip,width=58mm]{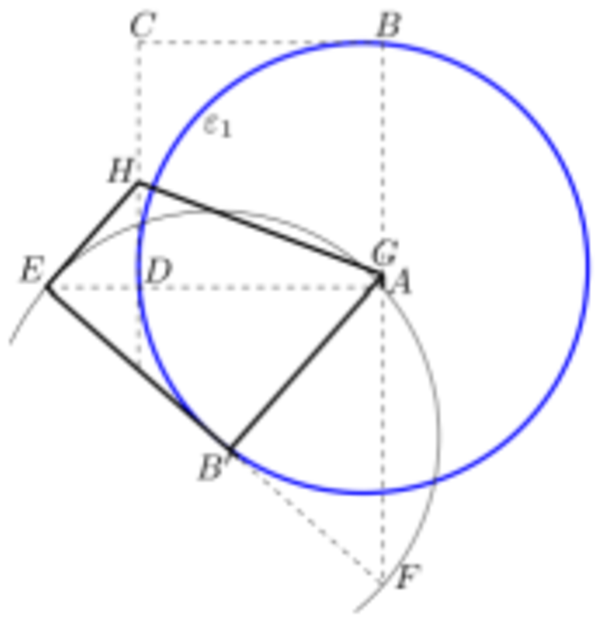}\refstepcounter{num}\label{fe1}\\
Figure \Fg : (h3) 
\end{center}  
\end{minipage} 

\medskip
\begin{minipage}{0.5\hsize}
\begin{center} 
\includegraphics[clip,width=55mm]{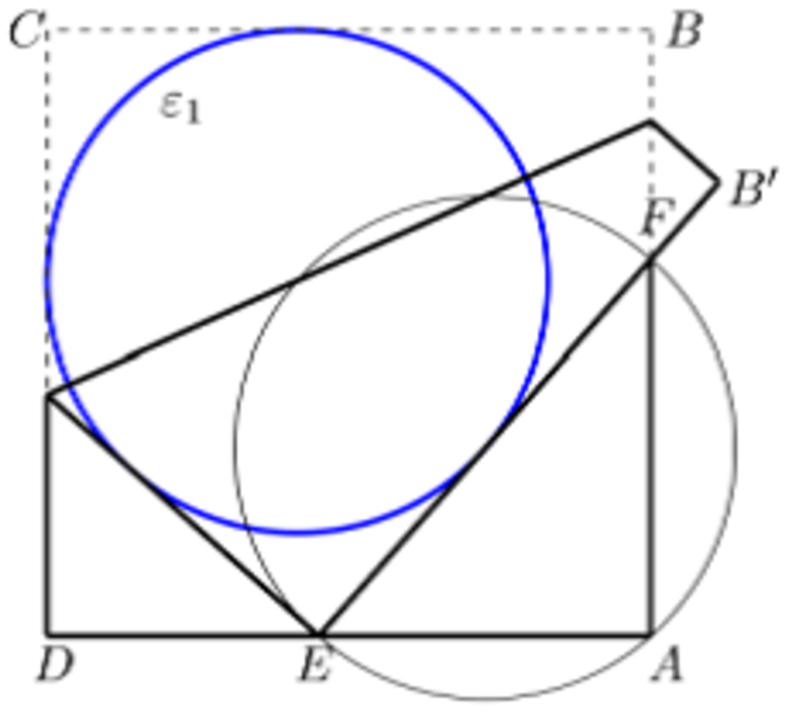}\refstepcounter{num}\label{fe2}\\
\medskip
Figure \Fg : (h5) 
\end{center}  
\end{minipage}
\begin{minipage}{0.5\hsize}
\begin{center} 
\includegraphics[clip,width=58mm]{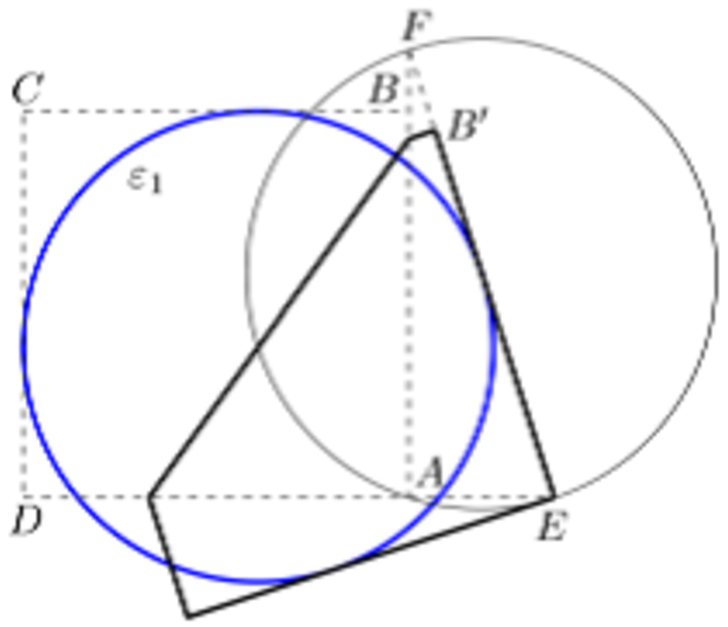}\refstepcounter{num}\label{fe3}\\
Figure \Fg : (h7) 
\end{center} 
\end{minipage} 

\begin{thm}\label{tcircm} The followings are true for $\CH$. \\
\noindent{\rm (i)} The relation $2r_1=|EF|$ holds. \\
\noindent{\rm (ii)} The circle $\Ep_1$ touches $EF$ at its midpoint. 
\end{thm}

\begin{proof} 
We prove (i). The degenerate cases are obvious. 
Assume that the lines $BC$ and $EF$ meet in a point $K$ and $x=|BK|$.  
We prove the case (h3) (see Figure \ref{fe1prf}). Since the triangles 
$BKF$ and $AEF$ are similar, $x/|BF|=|AE|/|AF|$, i.e., 
$x=c|AE|/|AF|$. Therefore by Theorem \ref{teqs}(ii), 
$$
|CK|=x-d=\frac{c|AE|-d|AF|}{|AF|}=\frac{c(b+d)-d(c-d)}{|AF|}
=\frac{bc+d^2}{|AF|}=\frac{d(a+d)}{|AF|}.  
$$ 
From the similar triangles $CKB'$ and $AEF$, we have $|CK|/r_1=|AE|/a$. 
Therefore by Theorems \ref{ghaga}, \ref{teqs} and Proposition 
\ref{pt}(ii), we get 
$$
r_1=\frac{|CK|}{|AE|}a=\frac{d(a+d)a}{|AE||AF|}=\frac{ad(c-b)}{2bc}=
\frac{c-b}{2}=\frac{|EF|}{2}. 
$$
The rest of (i) can be proved in a similar way. We prove (ii). 
It is sufficient to consider the ordinary cases. The perpendicular to 
$EF$ at $E$ touches $\Ep_1$, for it is the reflection of the line 
$CD$ in $m$ and $CD$ touches $\Ep_1$. Therefore the perpendicular to 
$EF$ at $F$ also touches $\Ep_1$ by (i). Hence $\Ep_1$ touches 
$EF$ at its midpoints. 
\end{proof}

\medskip\medskip
\begin{center} 
\includegraphics[clip,width=60mm]{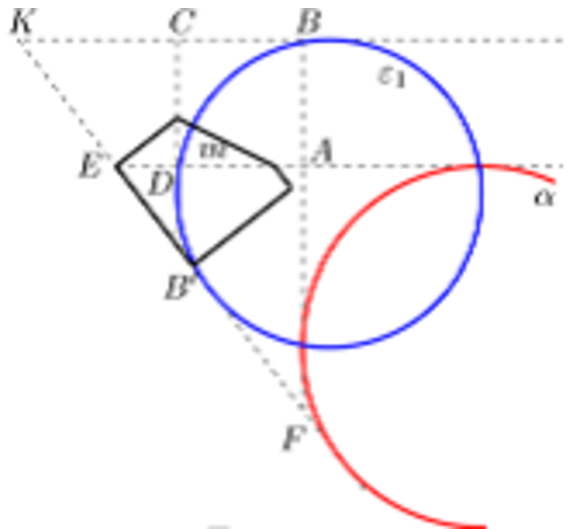}\refstepcounter{num}\label{fe1prf}\\
\vskip-2mm
Figure \Fg . 
\end{center}  

\section{Incircles and excircles of another triangles}\label{san}

Let us assume that the crease line $m$ meets the lines $AB$ and $CD$ in 
points $G$ and $H$, respectively for $\CH$. In this section we consider 
the incircles and the excircles of the triangles $B'FG$ and $DEH$. 
The sum of the inradii of the triangles $B'FG$ and $DEH$ equals $a$ in 
the case (h5) \cite{kkb}, \cite{Tkd} (see Figure \ref{feqs0}). We 
generalize this fact and give several new results. 

\medskip
\begin{center} 
\includegraphics[clip,width=55mm]{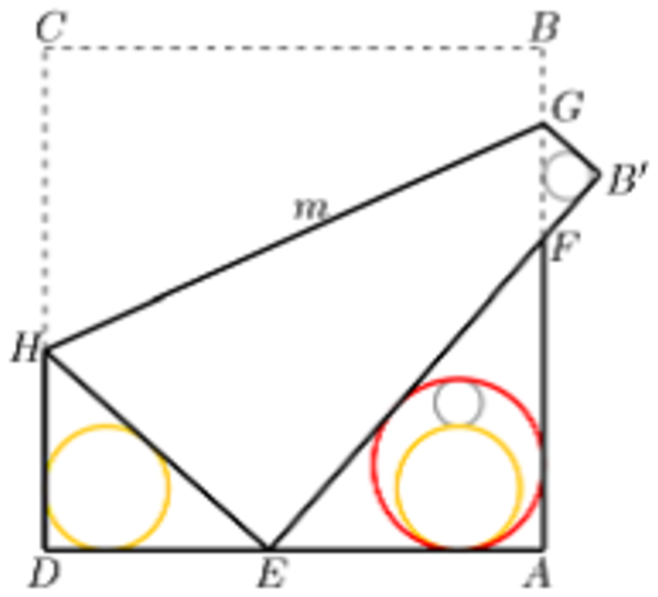}\refstepcounter{num}\label{feqs0}\\
Figure \Fg .
\end{center}

\begin{prop}\label{ppdh} The following relation holds for $\CH$. \\
$a=\left\{
\begin{array}{ll}\nonumber
|FG|-|DH| & \hbox{in the cases}\ \hbox{\rm (h1), (h3)}, \\
|DH|-|FG| & \hbox{in the cases}\ \hbox{\rm (h4), (h5),\ (h6), (h7)}.\\
\end{array} \right. \\
$
\end{prop}

\begin{proof} 
It is sufficient to consider the ordinary cases. We prove the cases 
(h1) and (h3) (see Figures \ref{fprfh1} and \ref{fprfh3}). Let $T$ be 
the point of contact of the circle $\De$ and the line $B'E$. The lines 
$DT$ and $m$ are parallel, since $DT$ is perpendicular to $CE$. Let $J$ 
be the point of intersection of the lines $AB$ and $DT$. Then the 
triangles $DAJ$ and $CDE$ are congruent. Hence $|AJ|=|DE|=|ET|$. 
Therefore we get $|B'T|=|BJ|$, while $T$ is the reflection of $B$ in 
the line $CF$. Therefore  $J$ is the reflection of $B'$ in $CF$. Hence 
$|DH|=|GJ|=|FG|-|FJ|=|FG|-|B'F|=|FG|-a$. The other cases can be proved 
similarly. 
\end{proof}

\begin{minipage}{0.5\hsize} 
\begin{center} 
\includegraphics[clip,width=58mm]{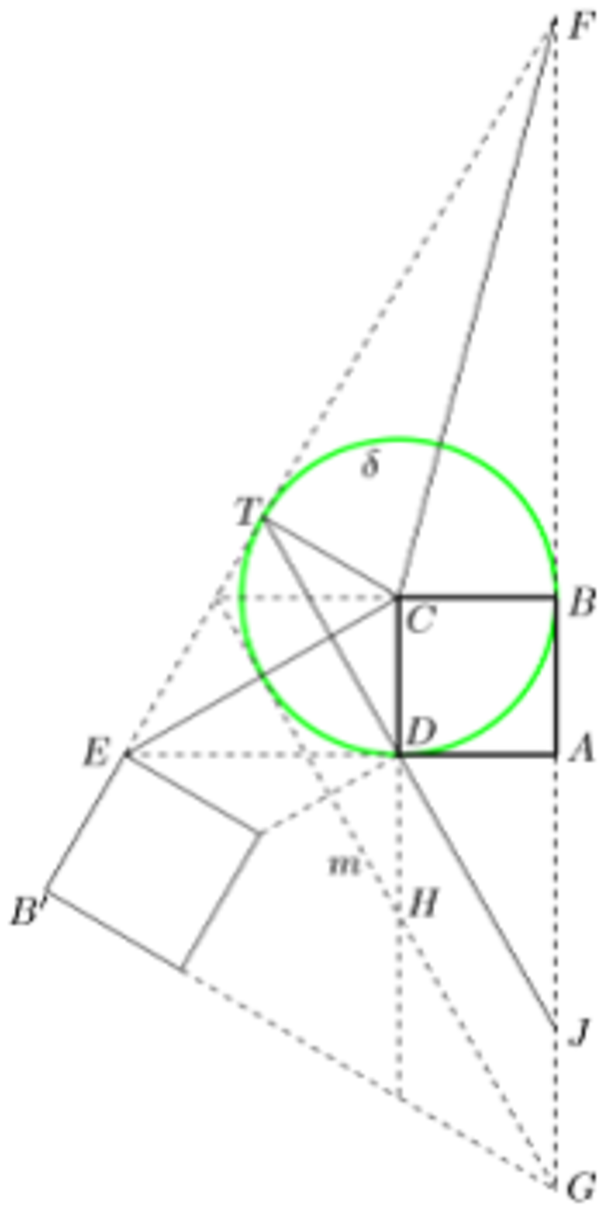}\refstepcounter{num}\label{fprfh1}\\
\vskip-2mm
Figure \Fg : (h1)
\end{center} 
\end{minipage}
\begin{minipage}{0.5\hsize} 
\begin{center} 
\vskip25mm
\includegraphics[clip,width=69mm]{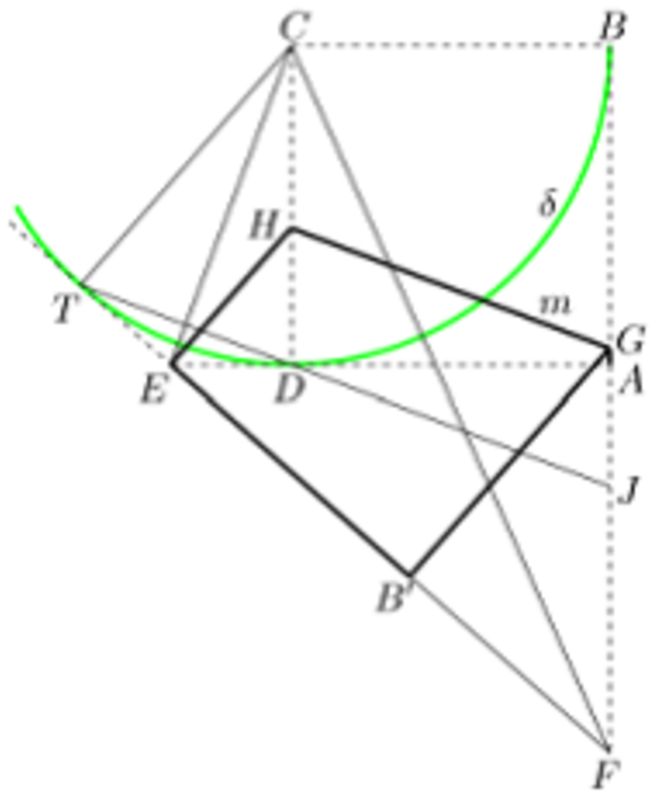}\refstepcounter{num}\label{fprfh3}\\
\vskip-2mm
Figure \Fg : (h3)
\end{center}  
\end{minipage}

\medskip
We define three circles (see Figures \ref{feqs_h1} to \ref{feqs_h7}). 
Let $\Ep_2$ be the excircle of $B'FG$ touching $B'F$ from the side 
opposite to $G$ in the cases (h1) and (h3), the incircle of $B'FG$ 
in the cases (h5) and (h7), and the point $A$ (resp. $B$) in the case 
(h4) (resp. (h6)).
Let $\Ep_3$ be the incircle or one of excircles of $B'FG$ touching 
$B'F$ from the side opposite to $\Ep_2$ in the ordinary cases, and  
the point $A$ (resp. $B$) in the case (h4) (resp. (h6)).
Let $\Ep_4$ be the excircle of $DEH$ touching $DE$ from the side 
opposite to $H$ in the cases (h1) and (h3), the point $D$ in the 
degenerate cases, and the incircle of $DEH$ in the cases (h5) and 
(h7). Let $r_i$ $(i=2,3,4)$ be the radius of $\Ep_i$. 

\begin{thm} \label{t567}
The following relations are true for $\CH$.\\
\noindent{\rm (i)} $r_3=r_4$. \hskip45mm {\rm (ii)} $a=r_2+r_4$. 
\end{thm}

\begin{proof} 
We prove (i). It is sufficient to prove the ordinary cases. In the 
case (h1), we have  
$$
r_3=\frac{|B'F|+|B'G|-|FG|}{2}=\frac{a-(|FG|-|BG|)}{2}=\frac{a-c}{2} 
$$
and
$$
r_4=\frac{|DE|-|DH|+|EH|}{2}=\frac{b-|DH|+d+|DH|}{2}=\frac{b+d}{2}.  
$$
Therefore $r_3=r_4$ by Theorem \ref{teqs}(i). 
In the case (h3), we have 
$$
r_3=\frac{|B'F|+|B'G|-|FG|}{2}=\frac{a+c-2|FG|}{2} 
$$
and 
$$
r_4=\frac{|DE|-|DH|+|EH|}{2}=\frac{b-2|DH|+d}{2}. 
$$
Therefore by Theorem \ref{teqs}(i) and Proposition \ref{ppdh}, 
$$
r_3-r_4=\frac{a+c-b-d-2|FG|+2|DH|}{2}=a-|FG|+|DH|=0.
$$
The rest of (i) can be proved similarly. 
By (i) and Theorem \ref{tgz}(iii), $a=r_2+r_3=r_2+r_4$. This proves (ii). 
\end{proof}

\vskip5mm
\begin{minipage}{0.47\hsize} 
\begin{center} 
\includegraphics[clip,width=66mm]{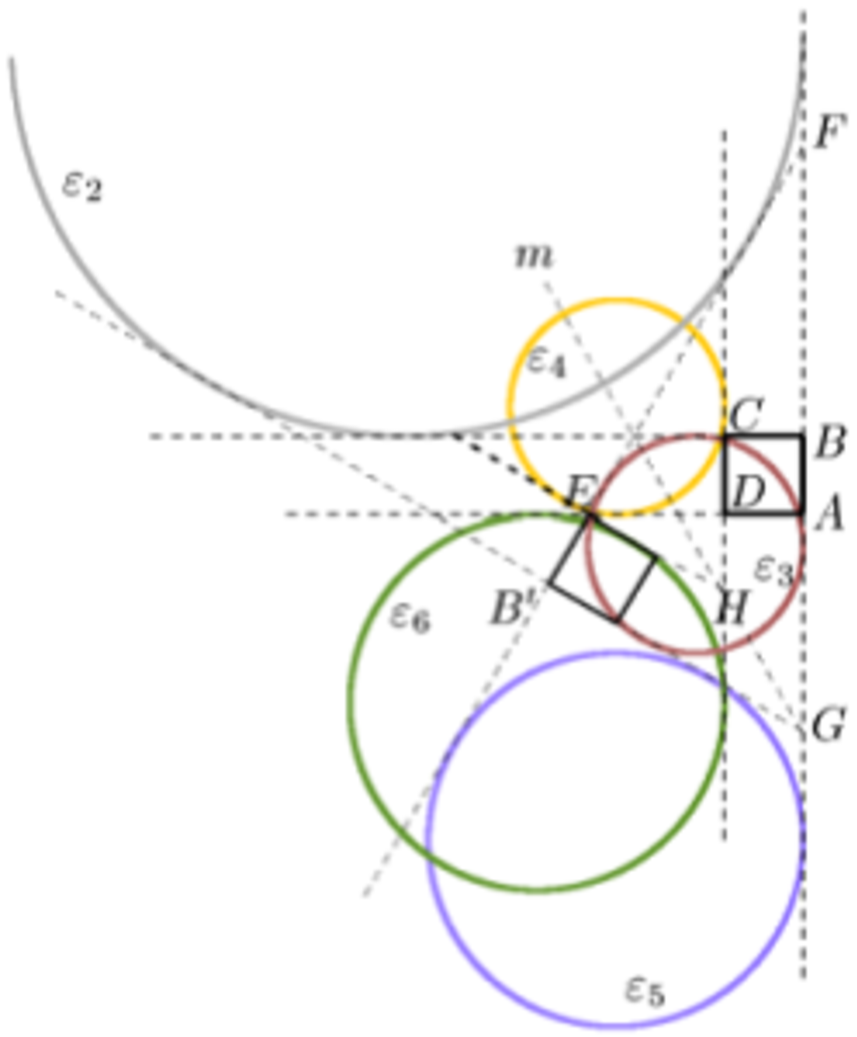}\refstepcounter{num}\label{feqs_h1}\\
Figure \Fg : (h1)
\end{center}  
\end{minipage}
\begin{minipage}{0.56\hsize} 
\begin{center} 
\vskip-5mm
\includegraphics[clip,width=68mm]{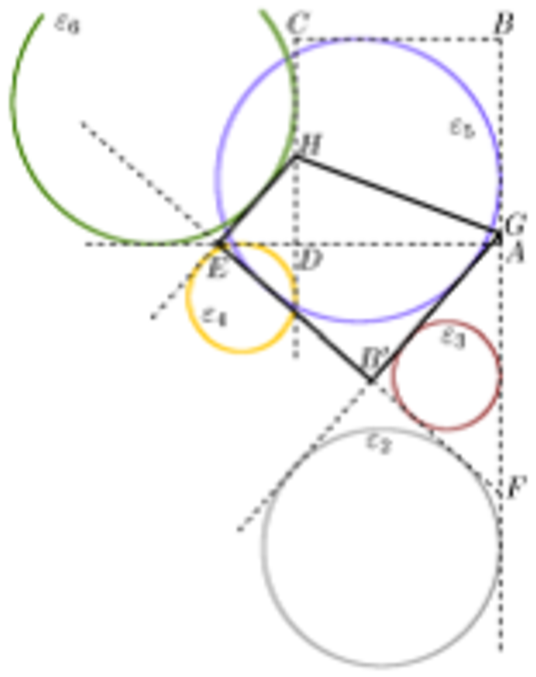}\refstepcounter{num}\label{feqs_h3}\\
Figure \Fg : (h3)
\end{center}  
\end{minipage}

\medskip
\begin{minipage}{0.55\hsize} 
\begin{center} 
\vskip24mm
\includegraphics[clip,width=62mm]{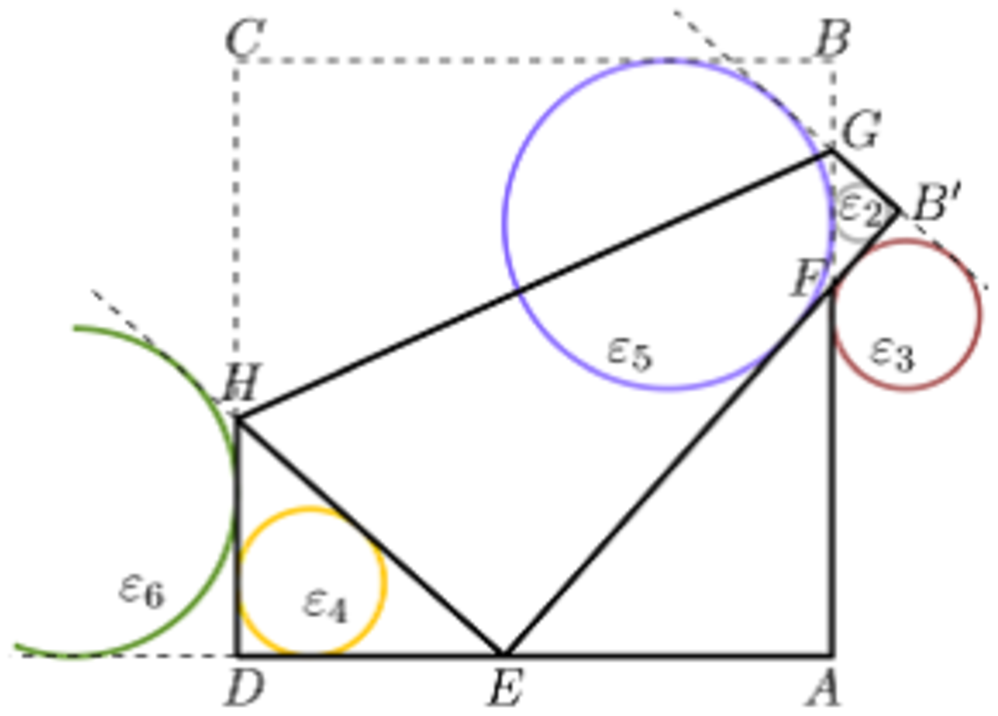}\refstepcounter{num}\label{feqs_h5}\\
\vskip9mm
Figure \Fg : (h5)
\end{center}  
\end{minipage}
\begin{minipage}{0.45\hsize} 
\begin{center} 
\includegraphics[clip,width=56mm]{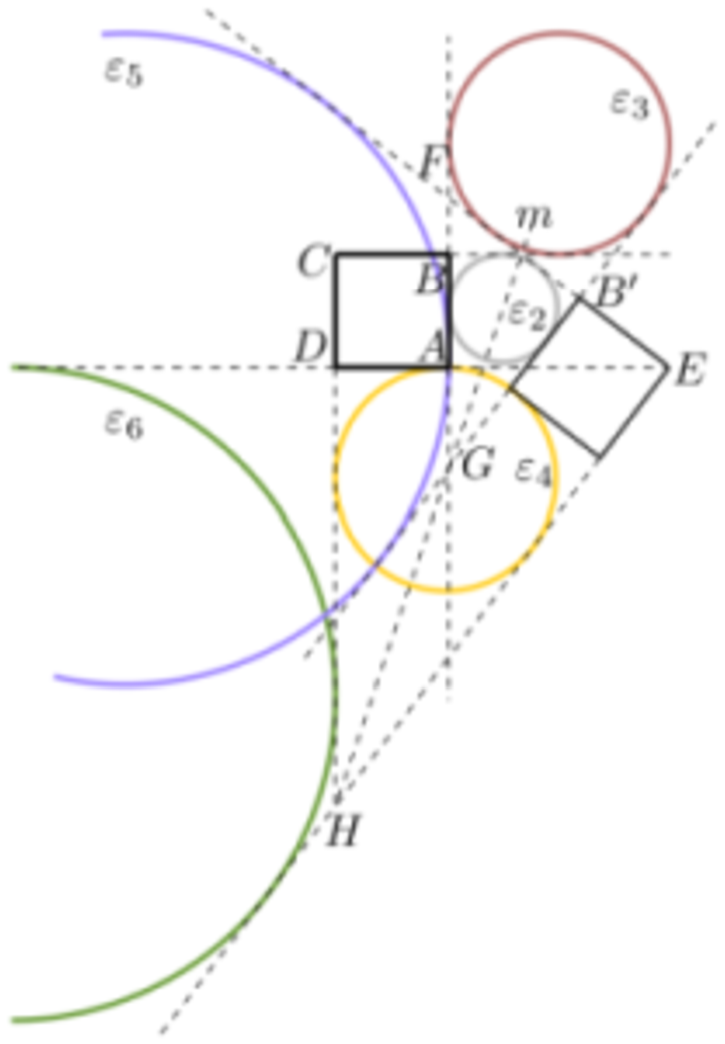}\refstepcounter{num}\label{feqs_h7}\\
\vskip1mm
Figure \Fg : (h7)
\end{center}  
\end{minipage}

We consider two more circles (see Figures \ref{feqs_h1} to 
\ref{feqs_h7}). The circle $\Ep_5(\not=\Ep_3)$ is defined as follows: 
The incircle or one of the excircles of $B'FG$ with center lying on the 
line joining $F$ and the center of $\Ep_3$ in the ordinary cases. The 
incircle of $ABCD$ in the case (h4). The point $B$ in the case (h6). 
Also the circle $\Ep_6(\not=\Ep_4)$ is defined as follows: The incircle 
or one of the excircles of $DEH$ with center lying on the line joining 
$E$ and the center of $\Ep_4$ in the ordinary cases. The reflection of 
the incircle of $ABCD$ in the line $CD$ in the case (h4). The point $D$ 
in the case (h6). The proof of the next theorem is similar to that of 
Theorem \ref{t567}(i) and is omitted.

\begin{thm}
The circles $\Ep_5$ and $\Ep_6$ are congruent. 
\end{thm}

Let $t$ be the remaining common tangent of $\Ep_3$ and $\Ep_5$ in 
the ordinary cases. Then $t$ is perpendicular to $AB$ by Theorem 
\ref{tgz}(i). Hence the triangle formed by the lines $B'G$, $AB$ and 
$t$ is congruent to $DEH$. Therefore the figure consisting of $\Ep_3$ 
and $\Ep_5$ is congruent to the figure consisting of $\Ep_4$ and $\Ep_6$. 

\section{Conclusion}

Haga's fold is generalized, and the recent generalization of Haga's 
theorems holds for the generalized Haga's fold, which is derived as a 
very special case from a theorem on three tangents of a circle. 
There are three similar right triangles in the configuration made by the 
generalized Haga's fold, and the incircles and the excircles of those 
triangles have many simple relationships.


\medskip
Hiroshi Okumura \\
Maebashi Gunma 371-0123, Japan \\
e-mail: \href{mailto:hokmr@protonmail.com}{hokmr@protonmail.com} \\

\end{document}